\documentclass[12pt, a4paper]{amsart}

\usepackage[utf8]{inputenc}
\usepackage{amsmath,amsthm,amsfonts,amssymb}
\usepackage{mathtools}
\usepackage[all]{xy}
\usepackage{tikz-cd}
\usepackage[shortlabels]{enumitem}
\usepackage{verbatim}
\usepackage{multirow}
\usepackage{tikz}
\usetikzlibrary{matrix}
\usetikzlibrary{cd}
\usepackage[hidelinks]{hyperref}

\usepackage{float}

\newcommand{\T}{{\mathbb{T}}}

\newcommand{\Kp}{K^{\perp}/K}

\newcommand{\R}{{\mathbb{R}}}
\newcommand{\C}{{\mathbb{C}}}

\newcommand{\J}{\mathcal{J}}

\newcommand{\ra}{\rangle}
\newcommand{\la}{\langle}

\newcommand{\Diff}{\mathrm{Diff}}

\newcommand{\lie}{\mathcal{L}}
\newcommand{\pr}{\mathrm{pr}}
\newcommand{\TM}{\mathbb{T}M}
\newcommand{\TCM}{\mathbb{T}_\C M}
\newcommand{\TCN}{\mathbb{T}_\C N}

\newcommand{\TN}{\mathbb{T}N}

\theoremstyle{plain}
\newtheorem{theorem}{Theorem}[section]
\newtheorem{proposition}[theorem]{Proposition}
\newtheorem{corollary}[theorem]{Corollary}
\newtheorem{lemma}[theorem]{Lemma}

\theoremstyle{remark}
\newtheorem{remark}[theorem]{Remark}

\newtheorem{example}[theorem]{Example}

\theoremstyle{definition}
\newtheorem{definition}[theorem]{Definition}

\newtheorem{innercustomthm}{Theorem}
\newenvironment{customthm}[1]
{\renewcommand\theinnercustomthm{#1}\innercustomthm}
{\endinnercustomthm}

\DeclareMathOperator{\type}{type}

\DeclareMathOperator{\order}{order}
\DeclareMathOperator{\Ann}{Ann}
\DeclareMathOperator{\Aut}{Aut}
\DeclareMathOperator{\codim}{codim}
\DeclareMathOperator{\im}{Im}
\DeclareMathOperator{\re}{Re}
\DeclareMathOperator{\rk}{rank}
\DeclareMathOperator{\cork}{corank}

\DeclareMathOperator{\Id}{Id}

\newcommand{\oL}{{\overline{L}}}

\newcommand{\st}{\;|\;}

\newcommand{\cJ}{\mathcal{J}}

\newcommand{\al}{\alpha}
\newcommand{\be}{\beta}

\renewcommand{\phi}{\varphi}

\DeclareMathOperator{\End}{End}

\newcommand{\oE}{\overline{E}}
\DeclareMathOperator{\ri}{real-index}

\title[Complex Dirac structures: invariants and local structure\;\;\;]{Complex Dirac structures: \\ invariants and local structure}
\author[D. Aguero]{Dan Aguero}
\author[R. Rubio]{Roberto Rubio}

\address{Instituto de Matemática Pura e Aplicada (IMPA), Rio de Janeiro  22460-320, Brazil}
\email{dagack2918@gmail.com}

\address{Universitat Aut\`onoma de Barcelona, 08193 Barcelona, Spain, and \\ \indent Universitat de Barcelona, 08007 Barcelona, Spain}
\email{roberto.rubio@uab.es}

\thanks{This project has received funding from the European Union’s Horizon 2020 research and innovation programme under the Marie Sklodowska-Curie grant agreements No 750885 GENERALIZED and No 801370 (COFUND Beatriu de Pinós 2018 00195). The first author has been supported by a PhD scholarship of the CNPq, and by the CNE Faperj. The second author has been supported by the Spanish MICINN under grant PID2019-109339GA-C32.}

\begin{document}
\begin{abstract}
We study complex Dirac structures, that is, Dirac structures in the complexified generalized tangent bundle. These include presymplectic foliations, transverse holomorphic structures, CR-related geometries  and generalized complex structures. 
We introduce two invariants, the order and the (normalized) type. We show that, together with the real index, they allow us to obtain a pointwise classification of complex Dirac structures. For constant order, we prove the existence of an underlying real Dirac structure, which generalizes the Poisson structure associated to a generalized complex structure. For constant real index and order, we prove a splitting theorem, which gives a local description in terms of a presymplectic leaf and a small transversal. 
\end{abstract}
\maketitle

\section{Introduction}

A Poisson bivector does not necessarily restrict on a submanifold to a Poisson bivector, but to a Dirac structure \cite{courant1990dirac}. Analogously, generalized complex structures \cite{hitchin:2004,gualtieri:2007}, encompassing complex and symplectic structures, fail to restrict on a submanifold as a generalized complex structure. For instance, a codimension-one submanifold of a complex manifold does not inherit a generalized complex structure, but a CR structure; whereas a codimension-one submanifold of a symplectic manifold inherits a presymplectic structure that is necessarily degenerate, and hence not generalized complex. However, both structures define complex Dirac structures on the submanifolds. Thus, just as Dirac structures arise when studying submanifolds of Poisson manifolds, complex Dirac structures naturally appear on submanifolds of generalized complex manifolds.  As generalized complex geometry has been proved to be useful in the understanding of supersymmetry \cite{lindstrom2005generalized,hull-lindstrom}, sigma models \cite{calvo2006supersymmetric,kapustin2007topological}, and field theory and the AKSZ formalism \cite{cattaneo20102d}, we expect complex Dirac structures to play an important role in these and related topics.

Complex Dirac structures are Dirac structures in the complexified generalized tangent bundle. The real index is the most basic invariant of a complex Dirac structure \cite{kopzcynski-trautman, gualtieri:2004}. When the real index is zero, a complex Dirac structure is a generalized complex structure. For arbitrary real index, complex Dirac structures can also describe objects such as presymplectic structures and CR-related geometries, extending the way generalized complex geometry encompasses symplectic and complex structures.

The purpose of this work is a systematic study of complex Dirac structures. In this paper, we introduce two new invariants: the order (Definition \ref{def:order}), which vanishes on generalized complex structures, and a normalized version of the type (Definition \ref{def:type}), which extends the type for generalized complex structures. For fixed real index and order, structures with type $0$ or maximal are, respectively, transformations of regular real Dirac structures and transverse CR structures. We prove that the three invariants altogether determine the pointwise geometry of a complex Dirac structure (Proposition \ref{prop:linear-normal-form}).

Since complex Dirac structures include generalized complex structures, 
it is natural to ask what properties can be extended and how. We focus on two of them. On the one hand, any generalized complex structure has associated a Poisson structure. In the context of complex Dirac structures we prove the following:
\begin{customthm}{\ref{th:Dirac-constant-order}}\label{thm-intro-1}
A complex Dirac structure with constant order has an underlying real Dirac structure, which agrees with the Poisson structure of a generalized complex structure when the real index is zero.
\end{customthm}
\noindent This suggests that the order is a more natural invariant than the real index. Indeed, we provide an example of a complex Dirac structure with constant real index whose presymplectic distribution is not smooth and hence does not define a real Dirac structure (Section \ref{sec:ex-order-change}).

On the other hand, generalized complex structures \cite{abouzaid2006local} and real Dirac structures \cite{blohmann} have a splitting theorem inspired by the Weinstein splitting theorem for Poisson structures. For complex Dirac structures we prove the following:
\begin{customthm}{\ref{splitting}}\label{ref:theo-intro-2}
Let $L$ be a complex Dirac structure with constant real index $r$ and order $s$, and let $m\in M$ be a point of type $k$.
Then, locally around $m$, $L$ is equivalent (via a diffeomorphism and $B$-transformation) to the product of a presymplectic manifold (with $(r-s)$-dimensional kernel) and a complex Dirac structure of constant real index $s$ and order  $s$ whose associated real Dirac structure is the graph of a Poisson bivector vanishing at $m$. 
\end{customthm}

 When the real index and the order vanish, we recover the splitting theorem for generalized complex structures \cite[Thm. 1.4]{abouzaid2006local}. The proof of our result relies on the techniques developed in \cite{bursztyn2019splitting}.

 Among other applications, we hope that this approach to complex Dirac structures will allow us  to better understand reduction \cite{bursztyn2007reduction, hu2009hamiltonian, lin2006symmetries,  stienon2008reduction, vaisman2007reduction} and blow-up \cite{ bailey2019blow, cavalcanti2009blow, cavalcanti2011blowing,  van2018blow} of generalized complex structures, as we can drop the hypothesis of the resulting structure remaining generalized complex.

The article is organized as follows. In Section 2, we recall the basic definitions and properties of Dirac and generalized complex structures. Section 3 deals with submanifolds of generalized complex structures and serves as a motivation. In Section 4, we introduce the main tools of the theory: the associated distributions and the two new invariants, the order and the (normalized) type. We present the relationship among these invariants and a pointwise classification of complex Dirac structures. In Section 5, we prove Theorem \ref{thm-intro-1} and study some of its consequences. Section 6 gathers two important examples:  a complex Dirac structure with constant real index  with no associated real Dirac structure, and a foliation by generalized complex leaves. Finally, in Section 7 we prove Theorem \ref{ref:theo-intro-2} and in an appendix we give a visual representation of the invariants.

\subsection*{Acknowledgements}
This work stems from the PhD thesis of the first author, supervised by the second author and Henrique Bursztyn, to whom we are greatly indebted.  The first author is also grateful to Hudson Lima for helpful conversations. We would also like to thank the anonymous referee for her/his careful reading of the paper.

\vfill

\subsection*{Notation and conventions} We denote by $M$ a smooth manifold. By a distribution we mean a subspace assignment $p\in M\mapsto R_p\subseteq E_p$, where $E$ is a vector bundle over $M$. 
A distribution is said to be smooth when any $v\in R_p$ can be extended to a local section of $E$ taking values in $R$. The rank of a distribution is the assignment $p\mapsto\dim R_p$. A distribution is said to be regular if it is smooth and its rank is constant. Regular distributions are actually subbundles, but we will keep the term `regular distribution' when $E$ is $TM$ or $TM_{\C}$ and use subbundle otherwise.
We will omit the vector bundle $E$ when it is clear from the context.

Given the complexification of a bundle, $E_\C$, and a distribution $L\subset E_\C$ such that $L=\overline{L}$, we denote the real elements of $L$ by
$$ \re L := L\cap E$$
and call them the real part of $L$. 
We will denote the complexification of a map with the same symbol.

\newpage

\section{Generalized Geometry}
\subsection{The generalized tangent bundle}\label{sec:gen-tgt}
We consider the  {\em generalized tangent bundle} $\TM:=TM\oplus T^*M$ with its natural nondegenerate symmetric bilinear pairing
$$\langle X+\xi, Y+\eta\rangle=\frac{1}{2}(\eta(X)+\xi(Y)),$$
 and the Dorfman bracket \cite{dorfman1987dirac} on $\Gamma(\TM)$
$$[ X+\xi, Y+\eta]=[X,Y]+\lie_{X}\eta-\imath_{Y}d\xi,$$
for $X+\xi$, $Y+\eta \in \Gamma (\TM)$. The tuple $(\TM,\la\cdot,\cdot\ra ,[\cdot,\cdot], \pr_{TM})$ has the structure of a {\em Courant algebroid} \cite{liu1995manin}.

The {\em automorphisms} of $\TM$ are bundle automorphisms $F$ of $\TM$ covering $f\in \Diff(M)$ such that, for $u$, $v\in\Gamma(\TM)$,
\begin{enumerate}[a)]
\item $f^{*}\la F(u),F(v)\ra=\la u,v\ra$,
\item $F[u,v]=[F(u),F(v)]$,
\item $\pr_{TM} \circ F=f_{*}\circ \pr_{TM}$.
\end{enumerate}
We denote the group of automorphisms of $\TM$ by $\Aut(\TM)$.

\begin{example} \label{ex:aut-Courant}
Any $f\in \Diff(M)$ defines $\T f\in\Aut(\TM)$ by
$$\mathbb{T} f(X+\xi)=f_{*}X+(f^{-1})^{*}\xi.$$
Any two-form $B\in \Omega^{2}(M)$ defines a bundle automorphism $e^{B}$ by
$$e^{B}(X+\xi)=X+\xi+\imath_{X}B.$$
When $B$ is closed, $e^{B}\in \Aut(\TM)$ and we call it a {\em $B$-transformation}. 
 \end{example}

These examples generate all the automorphisms \cite[Prop. 2.2]{gualtieri:2004}: 
\begin{align*}
    \Diff(M)\ltimes \Omega^{2}_{cl}(M) & \cong \Aut(\TM)\\
    (f,B) & \mapsto \mathbb{T} f\circ e^{B}.
\end{align*}
 
 The action of a generalized vector field $X+\xi\in \Gamma(\TM)$ via the Dorfman bracket $[X+\xi, \;\cdot\;]$ defines an endomorphism of $\Gamma(\TM)$. By \cite{hu2009hamiltonian}, it integrates to a one-parameter subgroup of automorphisms of $\TM$ given by $\{\T\varphi_{s}\circ e^{\gamma_{s}}\}_{s\in \R}$, with $\{\varphi_s\}_{s\in \R}$ the one-parameter subgroup integrating $X$ (with the convention $X=\frac{d}{dt}|_{t=0}\varphi^{*}_{-t}$) and  $$\gamma_{s}=\int^{s}_{0} \varphi_{u}^{*}(-d\xi) du.$$
This means that 
$$ [X+\xi,Y+\eta] = \frac{d}{dt}\bigg\rvert_{t=0} \left( (\T\varphi_{s}\circ e^{\gamma_{s}})(Y+\eta) \right).$$

The bundle $\TCM:=(\TM)_\C$ equipped with the complexification of the pairing and bracket of $\TM$ has analogous properties to the Courant algebroid $\TM$. Just as in Example \ref{ex:aut-Courant}, the map $f\in \Diff(M)$ gives rise to  $\T f\in \Aut(\mathbb{T}_{\C} M)$, and a closed  $B\in \Omega^2_\C(M)$ to $e^{B}\in \Aut(\TCM)$, which we call a {\em complex B-transformation}.

\subsection{Dirac structures}\label{sec:Dirac}
A {\em Dirac structure} \cite{courant1990dirac} is a lagrangian subbundle $L\subset \TM$ that is involutive with respect to the Dorfman bracket (that is, $[\Gamma(L),\Gamma(L)]\subseteq \Gamma(L)$).

\begin{example}\label{ex:presymp-Poisson}
\normalfont
 The graphs of a presymplectic structure $\omega$ and a Poisson structure $\pi$ are Dirac structures:
 \begin{align*}
 L_{\omega}&=\{X+\imath_{X}\omega\;\;|\;\; X\in TM\}, & L_{\pi}=\{\pi(\alpha)+\alpha\;\;|\;\; \al\in T^*M\}.
\end{align*} 

\end{example} 
The {\em range distribution} $E$ of a lagrangian subbundle $L$ is $E:=\pr_{TM}L$, which is smooth but not necessarily regular. There exists a  skew-symmetric bilinear map $\varepsilon:E\times E\to \R$ such that 
 $$L=L(E,\varepsilon):=\{X+\xi\:|\: X\in E,\:\:\xi_{|E}=\imath_{X}\varepsilon\}.$$
The range distribution of a Dirac structure is integrable and $\varepsilon$ restricts to each leaf as a presymplectic form  \cite{courant1990dirac}, so it generalizes the symplectic foliation associated to a Poisson bivector. Furthermore:
\begin{proposition}[\cite{dufour2008local}]\label{sing_Dirac}
Let $L$ be a Dirac structure and a point $m\in M$. If the presymplectic leaf passing through $m$ is a single point, then on a neighbourhood of $m$, $L$ is the graph of a Poisson structure. 
\end{proposition}
\noindent A Dirac structure is called {\em regular} when  $E$ is regular. We then have:
\begin{proposition}[\cite{courant1990dirac,gualtieri:2007}]\label{prop:E-regular} 
 A lagrangian subbundle $L(E,\varepsilon)$ is a regular Dirac structure if and only if $E$ is regular, $\varepsilon$ is a (smooth) bundle map and $d_E \varepsilon=0$, where $d_{E}$ denotes the differential along the directions of~$E$.
\end{proposition}

On the other hand, we call $L\cap TM$ the {\em null distribution} of a Dirac structure $L$. If $L\cap TM$ has constant rank, then it is smooth and integrable and its associated foliation is called the {\em null foliation}. 

Given a map $\varphi:N\to M$ and lagrangian subbundles $L\subset \TM$, $L'\subset \TN$, the {\em backward image} of $L$ is the (regular) distribution
$$\varphi^{!}L:=\{X+\varphi^{*}\xi \st \varphi_{*}X+\xi\in L\}\subset \TN,$$
and the {\em forward image} of $L'$ is the (regular) distribution 
$$\varphi_! L'=\{\varphi_* X+\xi\st X+\varphi^{*}\xi\in L'\}\subset \phi^*\TM.$$
For an inclusion $\iota:N\to M$, a sufficient condition for $\iota^! L\subset \TN$ to be a Dirac structure is the transversality  condition $\pr_{TM} L_{|N}+TN=TM_{|N}$ (see, for instance,  \cite{bursztyn-leyva}). 

\subsection{Generalized complex structures}\label{sec:gcs} A {\em generalized almost complex structure} \cite{gualtieri:2007} is a bundle map $\mathcal{J}:\TM\to \TM$
such that $\mathcal{J}^{2}=-1$ and $ \mathcal{J}^{*}=-\mathcal{J}$. If the lagrangian subbundle $\ker(\J-iId)\subset \TCM$ is involutive with respect to the Dorfman bracket we say that $\J$ is a {\em generalized complex structure}, which is equivalently given by an involutive lagrangian subbundle $L\subset \TCM$ such that $L\cap \overline{L}=0$.
\begin{example}\label{ex:LJ-Liw} A complex structure $J$ and a symplectic structure $\omega$ determine the generalized complex structures
\begin{align*}
 L_{J}&=T_{0,1}\oplus T_{1,0}^{*}, & L_{i\omega}=\{X+i\,\imath_{X}\omega\;\;|\;\; X\in T_\C M\}.
\end{align*} 
\end{example}
 A manifold admitting a generalized almost complex structure must be even dimensional, so we take $\dim M=2n$.

The main invariant of generalized complex structures is an integer-valued function called the {\em type}, which is defined, with the notation $E:=\pr_{T_\C M}L$, as
\begin{equation}\label{eq:type-gen-cplx}
\type L:=\cork E.    
\end{equation}
The type varies from $0$ to $n$. In Example \ref{ex:LJ-Liw}, the structure $L_{i\omega}$ is of constant type $0$, whereas $L_J$ is of constant type $n$. A point $m$ in $M$ is called of {\em complex type} if it is of type $n$, and it is called {\em regular} if there exists a neighborhood of $m$ with constant type.

The map $\cJ$ determines the Poisson bivector $\pi_{\J}=\pr_{TM}\J|_{T^{*}M}$, so a generalized complex structure gives a symplectic foliation. To recall Weinstein splitting-like theorems for generalized complex structures, we need the definition of product. Consider isotropic subbundles $K_{1}$ of $\TCM_1$ (or $\TM_1$) and $K_{2}$ of $\TCM_{2}$ (or $\TM_2$). Let $\pi_{i}:M_{1}\times M_{2}\to M_{i}$ denote the canonical projections. The bundle 
\begin{equation}\label{eq:prod-L1-L2}
    K_1\times K_2:=\pi_{1}^{*}K_{1}\oplus \pi_{2}^{*}K_{2}
\end{equation} is an isotropic subbundle of $\mathbb{T}_\C (M_{1}\times M_{2})$ (or $\mathbb{T}(M_{1}\times M_{2})$). For $L_1$ and $L_2$ generalized complex structures, the product $L_1\times L_2$ is a generalized complex structure.

By \cite{abouzaid2006local}, for any $m\in M$, there exists a neighborhood $U$, a closed two-form $B$, a symplectic structure $\omega$ and a generalized complex structure $L'$ of complex type at $m$ such that 
$$L|_{U}\cong e^{B}(L'\times L_{i\omega}).$$

Around a regular point $m$ of type $k$, we can be more precise  \cite{gualtieri:2004}:
$$L|_{U}\cong e^{B}(L_{J}\times L_{i\omega}),$$ 
where $J$ is the canonical complex structure on $\mathbb{C}^{k}$ and $\omega$ is the canonical symplectic structure of $\mathbb{R}^{2(n-k)}$.

\section{From generalized complex\\ to complex Dirac structures}

Let $M$ be a manifold with a generalized complex structure $L\subset \TCM$. Consider a submanifold $N\xhookrightarrow{\iota} M$. Under mild regularity conditions (see, e.g.,  Definition \ref{def:cplx-trans} below),  $\iota^{!}L\subset \TCN$ is a lagrangian and involutive subbundle. However, it is not necessarily a generalized complex structure, as the next examples show.
\begin{example}
Let $\omega\in \Omega^{2}(M)$ be a symplectic structure. Note that $\iota^{!}L_{i\omega}=L_{i\iota^{*}\omega}$ and so $\iota^{!}L_{i\omega}\cap \overline{\iota^{!}L_{i\omega}}=(\ker \iota^{*}\omega)_{\C}$, which is nonzero unless $N$ is a symplectic submanifold. 
\end{example}

\begin{example}
Let $J$ be a complex structure on $M$. Assume that $N$ is of codimension one and consider the $J$-invariant distribution $D:=TN\cap J(TN)$ over $N$. We have that $(D,J_{|D})$ is a CR structure of corank one in $N$, and $\iota^{!}L_{J}=L(\ker(J_{|D}-iId),0)$. Hence,  $\iota^{!}L_{J}\cap\overline{\iota^{!}L_{J}}=(\Ann D)_{\C}$, which has rank one.
\end{example}
 The subbundle $\iota^{!}L$ lies in a larger class than generalized complex:

\begin{definition}\label{def:cplx-Dirac} A {\em complex Dirac structure} is a lagrangian subbundle  $L\subset \TCM$ that is involutive (with respect to the Dorfman bracket).
\end{definition}
 \begin{example}\label{ex:complex-Dirac-1}
	A generalized complex structure is a complex Dirac 
structure $L$ such that $L\cap \overline{L}=\{0\}$. 	On the other hand, the complexification $L_\C$ of a real Dirac structure $L\subset \TM$  is a complex Dirac structure that satisfies $L_\C=\overline{L_\C}$ . The bundle $L_{i\omega}$ as in Example \ref{ex:LJ-Liw}  for a presymplectic form $\omega\in\Omega^2_{cl}(M)$ is a complex Dirac structure.
	\end{example}
	
	\begin{example}\label{ex:complex-Dirac-2}
    A CR-structure $(D,J)$, consisting of a regular distribution $D\subseteq TM$ and $J\in \End(D)$ such that $J^2=-\Id$, determines, with the notation $D_{1,0}:=\ker(J-i\Id)\subset D_\C$, the complex Dirac structure
    $$ L_{(D,J)} := D_{1,0} \oplus \Ann D_{1,0}.$$
\end{example}
 
\begin{proposition}\label{codim_1} Let $N\subseteq M$ be a submanifold of codimension $r$. Let $L$ be a lagrangian subbundle of $\TCM$ such that $L\cap \overline{L}=\{0\}$. Then $$\rk (\iota^{!}L\cap\overline{\iota^{!}L})\leq r.$$ 
\end{proposition}
\begin{proof}
Suppose that $\rk(\iota^{!}L\cap\overline{\iota^{!}L})>r$ at $m\in N$, so there exist real elements $$X_{1}+\xi_{1},\ldots , X_{r+1}+\xi_{r+1}\in \iota^{!}L_{m},$$
linearly independent, with $X_{j}\in T_{m}N$ and $\xi_{j}\in T^{*}_{m}N$. By definition of $\iota^{!}L$, there exist $\tau_{j}\in T^{*}_{m}M$ with $\iota^* \tau_{j}=\xi_{j}$ and $\eta_{j}\in \Ann T_{m}N$, such that
 $$X_{1}+\tau_{1}+i\eta_{1},\ldots ,X_{r+1}+\tau_{r+1}+i\eta_{r+1}\in L_{m}.$$
 Since $\dim \Ann T_m N=r$, there exist a non-trivial linear combination   $\sum^{r+1}_{j=1}c_{j}\eta_{j}=0$ with $c_j\in\R$. Consequently, $$\sum^{r+1}_{j=1}c_{j}(X_{j}+\tau_{j}+i\eta_{j})=\sum^{r+1}_{j=1}c_{j}(X_{j}+\tau_{j})$$ is real in $L_{m}$, so must vanish (since $L\cap \overline{L}=\{0\}$) and, thus,  $\sum^{r+1}_{j=1}c_{j}(X_{j}+\xi_{j})=0,$ which yields a contradiction.
\end{proof}
\begin{corollary} For a codimension-one submanifold $N \xhookrightarrow{\iota}M$ we have $$\rk(\iota^{!}L\cap\overline{\iota^{!}L})=1.$$
\end{corollary}
\begin{proof}
Since $\dim N$ is odd, then $\iota^{!}L\cap\overline{\iota^{!}L}\neq \{0\}$ and so the corollary follows from the proposition.
\end{proof}

\section{Invariants and classification\\ of complex Dirac structures}
In this section we introduce a set of invariants for complex Dirac structures which will allow us to characterize their local geometry. 

\subsection{The real index}\label{sec:real-index}

The following definition stems from \cite{kopzcynski-trautman}.
\begin{definition} The real index of any subbundle $L\subset \TCM$ is 
$$ \ri L := \rk(L\cap \oL),$$
which is a function that we usually denote by $r$.
\end{definition}

Let $L\subset \TCM$ be a lagrangian subbundle. Consider the distribution
$$ K:= \re ( L \cap \oL).$$
Its orthogonal distribution is given by $K^\perp = \re (L+\oL).$ We have $r=\ri L=\rk K$  and
 \begin{equation*}
 \left( \frac{\phantom{\perp}K^\perp\;}{K\;} \right)_\C \cong \frac{L + \oL}{L\cap \oL}. 
 \end{equation*}

The pairing on $\TM$ descends to a pointwise pairing on $K^\perp/K$ of signature $(\dim M-r,\dim M-r)$ (\cite[Prop. 2.28]{tiago:thesis}). When the real index $r$ is everywhere constant, the distributions $K$ and $K^\perp$ become bundles of rank $r$ and $2\dim M - r$, respectively. 
In this case, $K^{\perp}/K$ becomes a euclidean vector bundle (vector bundle with a nondegenerate pairing). In general, $K^{\perp}/K$ does not inherit a bracket from $\TM$. However, there are cases where $K^{\perp}/K$ can be reduced to a Courant algebroid \cite{bursztyn2007reduction,zambon2008reduction}. In this case, $L+L\cap \oL$ reduces to a complex Dirac structure, possibly with real index zero. In Section \ref{gen_fol}, we will see an example where $K^{\perp}/K$ is a Courant algebroid.

Many results, like the next two, are stated for lagrangian subbundles but are usually used for complex Dirac structures.

\begin{proposition}
\label{genmap}
A lagrangian subbundle $L\subset \TCM$ with constant real index $r$, is equivalent to the choice of an $r$-dimensional isotropic subbundle $K\subset \TM$ and a bundle map $\mathcal{J}:K^{\perp}/K \to K^{\perp}/K$ such that $\mathcal{J}^{2}=-1$ and $\mathcal{J}^{*}+\mathcal{J}=0$. 
\end{proposition}
\begin{proof}
 Given such $L$, the distribution $L_0=L+L\cap \oL\subset (L+\oL)/(L\cap \oL)$ is a lagrangian subbundle of $(\Kp)_{\C}$ with zero real index, that is, $L_{0}\cap \overline{L_{0}}=0$. So, $L_0$ defines a map $\J:\Kp\to \Kp$ such that $\J^{2}=-Id$. Moreover, since $L_0$ is lagrangian, $\J$ preserve the pairing and thus $\J^*+\J=0$.
 
Given an isotropic subbundle $K\subset\TM$ and a map $\J:\Kp\to\Kp$ satisfying the conditions of the statement, we retrieve a complex Dirac structure by taking  $L=q^{-1}(\ker(\J_{\C}-i\Id))$, where $q:K^{\perp}\to \Kp$ is the quotient map. Its real index equals $\rk K$ and  is  hence constant.
\end{proof}

\begin{corollary}\label{lemma:pointwisecJ}
A lagrangian subbundle $L\subset \TCM$ determines a distribution $K\subset \TM$ together with a pointwise complex structure $\mathcal{J}\in \End(K^{\perp}/K)$ which is moreover skew-symmetric, $\mathcal{J}^{*}+\mathcal{J}=0$.
\end{corollary}

 The map $\mathcal{J}$ in Corollary \ref{lemma:pointwisecJ} is pointwise a linear generalized complex structure. By the obstruction for their existence \cite[Prop. 4.5]{gualtieri:2004}:

\begin{proposition}
The dimension of a manifold $M$ admitting a lagrangian subbundle of $\TCM$ with real index $r$ must satisfy 
$$ \dim M \equiv r \mod 2.$$
Thus, the parity of the real index is constant.
\end{proposition}

\begin{example}
	For a generalized complex structure, the real index is zero and the associated map $\cJ$ corresponds to the generalized complex structure itself. For the complexification of a real Dirac structure, $L_\C$, the real index is $\dim M$ and the map $\cJ$ is zero. For $L_{i\omega}$ and $L_{(D,J)}$ as in Examples \ref{ex:complex-Dirac-1} and \ref{ex:complex-Dirac-2},
	\begin{align*}
	L_{i\omega}\cap \overline{L_{i\omega}}&=(\ker \omega)_\C,& \ri L_{i\omega}	&= \rk (\ker \omega),\\
	L_{(D,J)}\cap \overline{L_{(D,J)}} &= (\Ann D)_\C, & \ri L_{(D,J)} & = \cork D.
	\end{align*}

\end{example}

\subsection{The order and the (normalized) type}\label{sec:order-type}

 Analogously to the real case, a lagrangian subbundle $L\subset\TCM$ determines a complex range distribution and a skew-symmetric map $\varepsilon:E\to E^*$, 
\begin{align*}
E&:=\pr_{T_\C M} L,& \varepsilon& \in \wedge^2 E^*,
\end{align*}
such that $L=L(E,\varepsilon)$.

In order to describe lagrangian subbundles we shall associate real data, which can be interpreted geometrically. The following definitions associate real distributions to any complex distribution $E\subseteq T_\C M$, although we work on the case of a lagrangian subbundle $L$ with $E=\pr_{T_\C M} L$. Define the distributions
\begin{align*}
\begin{split}
\Delta &:= \re (E \cap \oE),\\
 D &:= \re ( E + \oE).
\end{split}
\end{align*}

Let $J:D/\Delta\to D/\Delta$ be the real part of the map on $(D/\Delta)_{\C}$ having $+i$-eigenspace $E/(E\cap \overline{E})$ and $-i$-eigenspace $\overline{E}/(E\cap \overline{E})$. The triple $(D,\Delta, J)$ is thus real data recovering $E$.

On the other hand, 
we consider the restriction of the two-form $\varepsilon$ on $\Delta$ and look at its imaginary part
$$ \omega_\Delta = \im \varepsilon_{|\Delta},$$
which is a possibly degenerate two-form $\omega_\Delta\in\wedge^2 \Delta^*$. Thus, at any $m\in M$, we have that $\Delta$ inherits a linear presymplectic two-form and $D/\Delta$ is endowed with a linear complex map.\footnote{We have ignored $\re\varepsilon$ since it will be regarded as a $B$-transformation.}

This motivates the introduction of new invariants involving the size of $D$ and $\Delta$, which we define in terms of both real and complex data.

\begin{definition}\label{def:order}
 The order of a lagrangian subbundle $L\subset \TCM$ is
 $$\order L := \cork D = \cork (E+\oE),$$
 which is a function that we usually denote by $s$.
\end{definition}
\begin{definition}\label{def:type}
 The type of a  lagrangian subbundle $L\subset \TCM$  is
 \begin{align*}
\type L :&=  \frac{1}{2} (\rk D - \rk \Delta)  \\
              &  =  \cork_{E+\oE} E,  
 \end{align*} 
which is a function that we usually denote by $k$.
\end{definition}
Note that, for a generalized complex structure, this definition coincides with the type as defined in \eqref{eq:type-gen-cplx}. We have the identity
\begin{equation}\label{eq:type-order}
\type L + \order L = \cork E.
\end{equation}

Finally, we see how the real index and the order are related. Define
$$\Delta_0:=\ker \omega_\Delta\subseteq TM.$$

\begin{lemma}\label{lemma:proj-K}
For the distributions $K$ and $K^\perp$ we have
\begin{align*}
    \pr_{TM} K^\perp &= D,& \pr_{TM} K &= \Delta_0.
\end{align*}
\end{lemma}
\begin{proof}
For the first part, $\pr_{TM} K^\perp=\pr_{TM}(L+\oL)=\pr_{TM}(E+\oE)= D$.

For the second part, take $X+\alpha \in K = \re (L\cap \oL)$, so $\alpha_{|E}=\varepsilon(X)$ and $\alpha_{|\oE} = \overline{\varepsilon}(X)$. When restricting to $\Delta$, $\varepsilon(X)_{|\Delta}=\overline{\varepsilon}(X)_{\Delta}$, so its imaginary part $\omega_\Delta(X)$ vanishes, that is, $X\in \Delta_0$. Conversely, for $X\in \Delta_0$, the forms $\varepsilon(X)\in E^*$ and $\overline{\varepsilon}(X)\in \overline{E}^*$ extend uniquely to a form  $\beta\in (E+\oE)^*$ such that $\overline{\beta}=\beta$. We extend it further to $\alpha\in T_\C^*M$ such that $\overline{\alpha}=\alpha$, that is, $\alpha \in T^*M$ so that $X+\al\in K$.
\end{proof}

\begin{lemma}\label{lemma:ri-order-ker}
For $L\subset \TCM$ a lagrangian subbundle, the distribution $K$ fits into the short exact sequence
\begin{equation}\label{eq:exact-seq-K}
 0\longrightarrow  \Ann D \longrightarrow K \xrightarrow[]{\pr_{TM}} \Delta_0 \longrightarrow  0
\end{equation}
and, consequently, 
\begin{equation}\label{eq:ri=or+rk-ker}
\ri L = \order L + \rk \Delta_0.
\end{equation}
\end{lemma}

\begin{proof}
We first look at $\ker \pr_{TM}|_{K}=K\cap T^*M$, which is $\Ann(pr_{TM} K^\perp)$, since $\alpha\in T^*M$ satisfies $\la \alpha, pr_{TM} K^\perp\ra = 0$ if and only if $\la \alpha, K^\perp\ra = 0$, that is $\alpha\in (K^\perp)^\perp=K$. By Lemma \ref{lemma:proj-K}, the  sequence follows.
\end{proof}

From \eqref{eq:ri=or+rk-ker}, it follows that 
\begin{equation}\label{eq:order-ri}
0\leq \order L \leq \ri L.    
\end{equation}

With the notation $r$ for real index, $s$ for order, $k$ for type, and $n$ for the function such that $\dim M=2n+r$, we have 
\begin{align}\label{eq:dim_assoc_dist}
     \rk D&=2n+r-s, \nonumber \\  \rk \Delta&=2(n-k)+r-s,\\  \rk \Delta_0&=r-s. \nonumber
\end{align}

\begin{example}
 For a generalized complex structure, the order is always zero and the type varies between $0$ and $\dim M/2$. For the complexification of a Dirac structure, $L_\C$, the order is $\cork \pr_{TM} L$ and the type is always zero. For $L_{i\omega}$ and $L_{(D,J)}$ as in Examples \ref{ex:complex-Dirac-1} and \ref{ex:complex-Dirac-2}, 
 \begin{align*}
 \order L_{i\omega} & = 0     ,& \order L_{(D,J)} & =\cork D,\\
 \type L_{i\omega} & = 0    ,& \type L_{(D,J)} & = n.
 \end{align*}
 \end{example}
 
 Equation \eqref{eq:ri=or+rk-ker} also generalizes the fact that $\omega_\Delta$ is nondegenerate for generalized complex structures \cite[Prop. 4.4]{gualtieri:2004}.

Regarding the product \eqref{eq:prod-L1-L2}, a direct computation shows the following.
\begin{lemma}\label{ri_order_adtv}
Let $L_{1}$, $L_{2}\subset \TCM$ be two lagrangian subbundles over the manifolds $M_{1}$ and $M_{2}$ with real parts $K_{1}$ and $K_{2}$, respectively. Then the real part of $L_{1}\times L_{2}$ is $K_1 \times K_2$ and, as a consequence,
\begin{align*}
\ri (L_1\times L_2)&=\ri L_1+\ri L_2,\\
\order (L_1\times L_2)&=\order L_1+\order L_2.
\end{align*}
\end{lemma}

\subsection{Properties of invariants and associated distributions}

We look now at the properties of the distributions and the three invariants introduced in Sections \ref{sec:real-index} and \ref{sec:order-type}. These properties just depend on the smoothness of the lagrangian subbundle.

To start with, the distributions $E$ and $D$ are smooth. Indeed, since $E$ is the image of $L$ under the anchor map, we have that $E$ is smooth. The distribution $\oE$ is then smooth and so $E+ \overline{E}$ is. Since $D=\pr_{TM} (E+\oE)$, we have that $D$ is also smooth.

Recall that the rank of a smooth distribution is a lower semicontinuous function (around any point the rank stays the same or drops). 

\begin{lemma}\label{lem:upper-semi}
	The real index and the order are upper semicontinuous. If the order is constant, the type is upper semicontinuous.
\end{lemma}

\begin{proof}
	The real index is $\rk(L\cap \oL) = 2\dim M - \rk (L+\oL)$, whereas the order is $\cork D$. Since $L+\oL$ and $D$ are smooth distributions, the first part follows. The second part follows similarly from \eqref{eq:type-order}. 
\end{proof}

The following result plays a key role in Section \ref{sec:assoc-Dirac}.

\begin{proposition}\label{prop:Delta-smooth}
	If the order is constant, the distribution $\Delta$ is smooth.
\end{proposition}

\begin{proof}
	Consider the imaginary part of the projection to $T_\C M$,
	\begin{equation*}
	\Phi=\im \circ\, \pr_{T_\C M} : L \to TM,
	\end{equation*}
	whose image is $\im(\pr_{T_\C M} L)=\im(E)=\im(E+\oE)=D$.

	Since the order is constant, $D$ is a vector bundle and hence  the kernel of $\Phi$,
	\begin{equation*}
	 L\cap (TM\oplus T_{\C}^* M),
	\end{equation*}
	is also a vector bundle. As $\pr_{TM}\ker\Phi=\Delta$, the proposition holds.
\end{proof}
We will see in Section \ref{sec:ex-order-change} an example where the order is not constant and $\Delta$ is not smooth.

Another distribution playing an important role in the theory is $K$. 
\begin{proposition}\label{K_lie_algebroid}
If $L$ has constant real index, then $K$ is a Lie algebroid whose orbits are precisely the leaves of $\Delta_0$.
\end{proposition}
\begin{proof}
This result follows from the involutivity of $K$ and Lemma \ref{lemma:proj-K}.
\end{proof}

\subsection{Pointwise description of complex Dirac structures}
\label{sec:pointwise}
The real index, order and type determine the pointwise geometry of the complex Dirac structure up to $B$-transform.

\begin{proposition}\label{prop:linear-normal-form}
 Let $L\subset \TCM$ be a lagrangian subbundle on a manifold $M$ of dimension $2n+r$. At a a point $m$ where
 \begin{align*}
 \ri L & = r, & \order L & = s, & \type L &=k,
  \end{align*}
there exists a complement $N$ of $\Delta$ in $T_m M$, and $B\in\wedge^2 T^*_m M$, such that, at the point $m$, 
$$ L\cong e^B(L_{i\omega_{\Delta}}\times L_{(C,J)}),$$ where
$L_{i\omega_{\Delta}}$ is associated to the two-form $\omega_{\Delta}\in \wedge^2 \Delta^*_m$, and $L_{(C,J)}$ is an almost CR-structure on $N$.
\end{proposition}

For its proof, we shall use the inclusion $\iota$ and the projection $p$ in 
\begin{equation*}
\begin{tikzcd} D \arrow[r, "p"]\arrow[d, "\iota"']& \dfrac{D}{\Delta_0} \\ TM & \end{tikzcd}\end{equation*}
together with the map $q:\Ann \Delta_0 \to (D/\Delta_0)^*$ such that $\iota^*=p^*\circ q$.

With sequence \eqref{eq:exact-seq-K} in mind, we say that $K$ splits when
$$ K =\Delta_0\oplus \Ann D.$$
We then have $K^\perp = D\oplus \Ann \Delta_0$ and, consequently,
$$\frac{\phantom{\perp}K^\perp\;}{K\;} \cong \frac{D}{\Delta_0}  \oplus \left( \frac{D}{\Delta_0} \right)^*,$$
where $[X+\xi] \in K^\perp/K$ maps to $pX+q\xi$. The pointwise generalized complex structure on $K^\perp/K$ (Corollary \ref{lemma:pointwisecJ}) corresponds  to $$\widetilde{L}:=\{pX+q\xi \st X+\xi \in L\},$$
whose associated symplectic subspace is pointwise  $\Delta/\Delta_0$.

\begin{proof}[Proof of Proposition \ref{prop:linear-normal-form}]
We work pointwise. We assume first that $K$ splits. With the notation above, we have that $\iota_!p^! \widetilde{L}=L$, since
$$\iota_!p^! \widetilde{L}=\{X+\xi \st \iota^*\xi=p^*\eta \textup{ for some } \eta \textup{ with }  pX+\eta\in \widetilde{L}\}$$
which clearly contains $L$ and they are both lagrangian. From the pointwise splitting for generalized complex structures (\cite[Thm. 3.6]{gualtieri:2007}), there exists a splitting $D/\Delta_0=\Delta/\Delta_0\oplus (C+\Delta_0)/\Delta_0$ and $B'\in\wedge^2 (D/\Delta_0)^*$ such that
$$\widetilde{L}\cong e^{B'}(L_{i\omega'}\times L_{J}).$$
We then have $p^! \widetilde{L}\cong e^{p^*B'}(L_{i\omega}\times L_J)$ on $D=\Delta\oplus C$ with $\omega=(p^*\omega')_{|\Delta}$ a presymplectic structure on $\Delta$, and $J$ a complex structure on $C\cong (C+\Delta_0)/\Delta_0$.  Finally, we look at $\iota_!p^!\widetilde{L}$: we choose any $N$ containing $C$, via $\iota^C:C\to N$, such that $TM=\Delta\oplus N$. For  $B\in \wedge^2 T^*M$ any extension of $p^*B'$ we have a natural injective map
 \begin{align*}
 e^{B}(L_{i\omega}\times \iota^C_! L_{J}) &\to \iota_!p^! \widetilde{L}\\
    e^B( X+\pi_\Delta^*i\omega(X)+Y+\pi_N^*\alpha )&\mapsto e^{p^*B'}(X+Y+i\omega(X)+\alpha).
 \end{align*}
Since they are both lagrangian subspaces, this map is an isomorphism. As, always at $m$, $\ker \omega=\Delta_0$, we have $L_{i\omega}\cong L_{i\omega_{\Delta}}$ so we conclude:
$$\iota_!p^! \widetilde{L}\cong e^{B}(L_{i\omega}\times \iota^C_! L_{J}) = e^{B}(L_{i\omega_\Delta}\times L_{(C,J)}),$$ with $\omega_\Delta\in \wedge^2\Delta^*_m$ and $(C,J)$ an almost CR-structure on $N$.

When $K$ does not split, sequence \eqref{eq:exact-seq-K} determines a map $$h:\Delta_0 \to T^*M/\Ann D\cong D^*$$  satisfying $h(x)(x)=0$ for $x\in\Delta_0$, as $K$ is isotropic. By taking a complement of $\Delta_0\subseteq D$ where we set the map to vanish, we get $h\in \wedge^2 D^*$, which we extend to $B\in \wedge^2 T^*M$. We then have that for $e^{-B}L$, its subbundle $K$  splits and we apply the first part of the proof.\end{proof}

\subsection{Extremal type}\label{sec:extremal} We study now complex Dirac structures with constant real index and order whose type is $0$ or maximal.

\begin{example}\label{ex:complexification-Dirac}
   A regular real Dirac structure   $L(\Delta,\omega)$, defines the complex Dirac structure
    $L(\Delta_\C,i\omega).$
    It has real index $r=\rk \ker \omega + \cork \Delta$, order $s=\cork \Delta$ and~type $0$. If $ \ker \omega$ is regular, $L(\Delta_{\C}, i\omega_{\C})$ has constant real index and order.
\end{example}

 This example is key to describe type $0$ complex Dirac structures.

\begin{proposition}\label{prop:type0}
 Let $L$ be a complex Dirac structure with constant real index $r$ and order $s$ over a $(2n+r)$-dimensional manifold. If $\type L=0$, then $L=e^{B}L(\Delta_{\C},i\omega)$,  with $\Delta$ a corank-$r$ involutive distribution, $\omega\in \wedge^{2}\Delta^{*}$ a $d_\Delta$-closed form with $(r-s)-$dimensional kernel, and $B$ a $d_{\Delta}$-closed real two-form (that is, $B|_{\Delta}$ is $d_{\Delta}$-closed). Locally, we can choose $B$ to be closed.
\end{proposition}
\begin{proof}
Let $L$ be $L(E,\varepsilon)$. Since $L$ has constant real index, order and type, $\Delta$ and $D$ are regular distributions. As $\type L=0$, the distribution $E$ is real and so $E=\Delta_{\C}$. Then $\varepsilon$ decomposes into its real and imaginary parts: $\varepsilon=B'+i\omega$, with $B'$, $\omega\in \wedge^2 \Delta^*$. By Proposition \ref{prop:E-regular}, we have that $B'$ is $d_\Delta$-closed,  $L(\Delta_{\C},i\omega)$ is involutive, and, for any extension $B$ of $B'$, $L=e^B L(\Delta_{\C},i\omega)$.
Locally, we obtain $B$ closed by considering a foliated chart and extending $B'$ to be constant in the directions of the distribution and vanish on the complementary directions.
\end{proof}

The maximal type, in dimension $2n+r$ with real index $r$,   is $n$. In order to describe this case, we introduce transverse CR structures\footnote{This name was used, with a different definition, in the context of foliations on CR-manifolds \cite{dragomir1997transversally}.}.
\begin{definition}\label{def:transverse-CR}
A transverse CR structure $(S,R,J)$ consists of regular distributions $R\subseteq S\subseteq TM$,  with $R$  integrable, and a bundle map $J:S/R\to S/R$, such that $J^2=-\Id$ and, for the projection $q:TM\to TM/R$, the distribution 
\begin{equation}\label{eq:E_J}
E=q^{-1}(\ker(J-i\Id)) \subset T_\C M 
\end{equation} 
is regular and involutive.
\end{definition}

Transverse CR structures include, as particular cases,  CR structures (when $R=0$) and transverse holomorphic structures (when $S=TM$). Given a transverse CR structure $(S,R,J)$, it can be proved \cite[Lem. 2.76]{aguero} that around any point,  there exists a neighbourhood $U$ where the foliation associated to $R_{|U}$ is simple and its leaf space carries a CR structure. 

    For a transverse CR structure $(S,R,J)$, with $E$ as in \eqref{eq:E_J}, we associate the complex Dirac structure \begin{equation}\label{eq:transverse-CR}
        L_{(S,R,J)}:=L(E,0)=E\oplus \Ann E,
    \end{equation}
    with real index $r=\cork S + \rk R$. If $\dim M=2n+r$, $L_{(S,R,J)}$ has order $s=\cork S$ and type $n$.

For maximal type complex Dirac structures, we prove the following:

\begin{proposition}\label{ext_type_cr}
Let $L$ be a complex Dirac structure with constant real index $r$ and order $s$ over a $(2n+r)$-dimensional manifold. If $L$ has constant type $n$, then there exist a transverse CR structure $(D, \Delta, J)$ and $B\in \Omega^2(M)$ such that $L=e^{B}L_{(D, \Delta,J)}$.
\end{proposition}

\begin{proof}
Let $L=L(E,\varepsilon)$ with $\varepsilon\in \Gamma(\wedge^2 E^*)$. Define a real form in $\Gamma(\wedge^2 (E+\overline{E})^*)$ by the skew-symmetric extension of 
\[\begin{cases}
\varepsilon(X) & \textrm{ for } X \in E,\\
\overline{\varepsilon(\overline{X})} & \textrm{ for } X \in \overline{E}.\\
\end{cases}\]
Note that it is well defined as type $n$ implies that $\Delta_0=\Delta$ and hence $\varepsilon(X)=\overline{\varepsilon(\overline{X})}$ for $X\in E\cap \overline{E}$. Extend it to $B\in \Omega^2(M)\subset \Omega^2_\C(M)$. We then have $L=L(E,\varepsilon)=e^BL(E,0)$.
We recall that the real data $(D,\Delta, J)$ recovers $E$, so 
$L(E,0)=L_{(D,\Delta,J)}$.
Finally,  $(D,\Delta,J)$  is a transverse CR-structure, since $\Delta$ is integrable because  $\Delta=\Delta_0=\pr_{TM} K$  and $K$ is a Lie algebroid (Lemma \ref{K_lie_algebroid}). 
\end{proof}

These results generalize the appearance in generalized complex structures of symplectic (type $0$) and complex (type  $\dim M/2$) structures.

\subsection{Pointwise classification}

Proposition \ref{prop:linear-normal-form} gives a pointwise classification of lagrangian subbundles, up to a linear $B$-transformation.

 Propositions \ref{prop:type0} and \ref{ext_type_cr} make the left and right columns of Table \ref{table:examples} work locally (with constant invariants) up to  $B$-transformation and real transformation, respectively. Corollary \ref{regular_splitting} below will make the central column work locally (up to a transformation closed in the presymplectic directions).

\begin{table}[H]\setlength{\tabcolsep}{3pt}
    \centering
\begin{footnotesize}
\noindent\begin{tabular}{|c|c|c|c|} 
\hline
\multirow{2}{*}{order $r$} &  symplectic $(\Delta,\omega_\Delta)$ & \multirow{2}{*}{$\cdots$} & CR structure $(D,J)$,  \\
&  $\cork \Delta=r$ & &   $\cork D=r$\\ \hline
\multirow{5}{*}{order $s$}  &  &  $(\Delta,\omega_\Delta)\times (D/\Delta,J)$  &  \\
    & presymplectic $(\Delta,\omega_\Delta)$    &   presymplectic $\times$ CR  & transverse CR $(D,\Delta,J)$ \\
    & $\cork \Delta=s$ & $\cork D=s$ & $\cork D=s$\\
    &  $\rk \ker \omega_\Delta=r-s$ & $\cork_D \Delta=2k$ & $\rk \Delta=r-s$\\
    & & $\rk \ker \omega_\Delta=r-s$ &\\ \hline
\multirow{2}{*}{order $0$} &  presymplectic $(TM,\omega)$ & \multirow{2}{*}{$\cdots$} & transverse holomorphic  \\
&   $\rk \ker \omega=r$ & &   $(TM,\Delta,J)$, $\rk \Delta=r$\\ \hline
&  type $0$ & type $k$ &  type $n$\\ \hline
\end{tabular}
\end{footnotesize}

\caption{Classification of complex Dirac structures}
\label{table:examples}
\end{table}

\section{The associated real Dirac structure}\label{sec:assoc-Dirac}

For a complex Dirac structure $L=L(E,\varepsilon)$, the distribution $\Delta$ is pointwise endowed with the presymplectic structure $\omega_\Delta$. When the order is constant, we have proved in Proposition \ref{prop:Delta-smooth} that $\Delta$ defines a smooth distribution. When $L$ is a generalized complex structure, $(\Delta,\omega_{\Delta})$ corresponds to the symplectic foliation of the associated Poisson structure. So it is natural to ask whether the distribution
$$\widehat{L}=L(\Delta,\omega_\Delta)$$
is a Dirac structure. In this section we prove that this is the case.

\begin{theorem}\label{th:Dirac-constant-order}
For constant order, the distribution $\widehat{L}$ is a real Dirac structure.
\end{theorem}

\begin{proof}
By definition, $\widehat{L}$ is pointwise lagrangian. We have to check that it is smooth and involutive. We first prove the following identity:
\begin{equation}\label{eq:widehatL}
 \widehat{L} = \pr_{\TM} \left( \begin{pmatrix}
1 & 0\\
0 & -i 
\end{pmatrix} ( L \cap (TM \oplus T_\C^*M)) \right),
\end{equation}
where the $2\times 2$ matrix denotes, by blocks, a linear map of $\TCM$. 

Take $X+\al\in \widehat{L}$. Denote also by $\alpha$ its $\C$-linear extension to $T_\C^*M$, and consider 
	$$ \varepsilon(X) - i \alpha_{|E} \in E^*.  $$
	We extend this element to $\gamma\in (E+\oE)^*$ by $\gamma(\overline{Z})=\overline{\gamma(Z)}$. This is well defined as it is real for $\Delta$, since
	$$\varepsilon(X)_{|\Delta} - i \alpha_{|\Delta} = \re \varepsilon(X)_{|\Delta}.$$
	The element $\gamma$ is the $\C$-linear extension of an element of $D^*$, which we can extend to $\al'\in T^*M$. We check that $X+\al'+i\al$ belongs to $L$:
	$$ (\al'+i\al)_{|E}=\varepsilon(X) - i\alpha_{|E}+i\alpha_{|E}=\varepsilon(X).$$
	Thus $\widehat{L}$ is contained in the right-hand side. The right-hand side is isotropic, as for $X+\al'+i\al$, $Y+\be'+i\be\in L$,
	$$ \la X+\al, Y+\be\ra = \im \la X+\al'+i\al, Y+\be'+i\be\ra = 0.$$	
	Since $\widehat{L}$ is lagrangian, they must be equal.
	
	To see the smoothness of $\widehat{L}$, note that in the proof of Proposition \ref{prop:Delta-smooth} we have shown that $L \cap (TM \oplus T_\C^*M)$ is smooth, and so its image by the linear map $\begin{psmallmatrix}
	1 & 0\\
	0 & -i 
	\end{psmallmatrix}$ and the projection $\pr_{\TM}$ is smooth. 
	
	Finally, note that the right-hand side of \eqref{eq:widehatL} is involutive, as the distribution $L \cap (TM \oplus T_\C^*M)$ is and the linear map and the projection preserve involutivity. Hence, $\widehat{L}$ is involutive and a Dirac structure.
\end{proof}
Constant order is a sufficient but not a necessary condition: let $L=L(S,\varepsilon)$ be a real Dirac structure, its complexification $L_{\mathbb{C}}$ is a complex Dirac structure. Note that $\widehat{L_{\mathbb{C}}}=L$ and $\order L_{\mathbb{C}}=\codim S$ as $D=S$. Hence, $L_{\mathbb{C}}$ has associated a real Dirac structure, although its order is not necessarily constant ($\codim S$ may change on real Dirac structures).

Since $\widehat{L}$ is a Dirac structure, its image by the anchor map is an integrable distribution, so we obtain a refinement of Proposition \ref{prop:Delta-smooth}:
\begin{corollary}
If the order is constant, then $\Delta$ is integrable.
\end{corollary}
We can thus associate a presymplectic foliation to any complex Dirac structure with constant order. We will see in Section 7 that this foliation plays a fundamental role in the splitting theorem for complex Dirac structures.
Note that constant real index does not play any role in the smoothness of $\Delta$ and hence $\widehat{L}$, as we shall see in Section \ref{sec:ex-order-change}.

\begin{example}
If we start with a real Dirac structure $L$, we have $\widehat{L_\C}=L.$
	For $L_{i\omega}$ as in Example \ref{ex:LJ-Liw} and $L_\omega$ as in Example \ref{ex:presymp-Poisson},
	$$ 	\widehat{L_{i\omega}} = L_\omega. $$
\end{example}

\begin{example}
	Starting with a real Dirac structure $L(\Delta,\omega_\Delta)$, 
	the complex Dirac structure $L:=L(\Delta_\C,i\omega_\Delta)$ (see Example \ref{ex:complexification-Dirac}) satisfies 
	$$\widehat{L}=L(\Delta,\omega_\Delta).$$
	For a CR-structure $(D,J)$ or, more generally, for a transverse CR-structure $(S,R,J)$ as in  \eqref{eq:transverse-CR}, we have
	\begin{align*}
	    \widehat{L_{(D,J)}} &= T^*M,& \widehat{L_{(S,R,J)}} &= R\oplus \Ann R.
	\end{align*}
\end{example}
\begin{remark}
From Theorem \ref{th:Dirac-constant-order}, constant order implies that $\Delta$ is an involutive distribution. As for $D$, non-Levi flat CR structures (e.g. $S^3$ with the CR structure inherited by $\C^2$) provide examples of complex Dirac structures with constant invariants but with non-involutive $D$.
\end{remark}
The associated Dirac structure $\widehat{L}$ is preserved by backward image.
\begin{lemma}\label{lem:widehat-backward}
Let $\psi:N\to M$ be a smooth map, and $L$ be a complex Dirac structure over $M$. Then, as distributions, $\psi^! \widehat{L}=\widehat{\psi^! L}$.
\end{lemma}
\begin{proof}
By equation \eqref{eq:widehatL},
\begin{align*}
 \widehat{\psi^! L} &= \pr_{\TN} \left( \begin{pmatrix}
1&0\\
0&-i
\end{pmatrix}\cdot \{X+\psi^* \xi_1+i\psi^* \xi_2\st \psi_* X+\xi_1+i\xi_2\in L\} \right)\\
 & =\{X+\psi^* \xi_2\st \psi_* X+\xi_2\in \widehat{L}\} =\psi^! \widehat{L}.
\end{align*}
\end{proof}
Let $N\xhookrightarrow{\iota} M$ be a submanifold. Just as in the real case (Section \ref{sec:Dirac}), for the backward image $\iota^! L$ to be a complex Dirac structure, it is analogously proved that a transversality condition is sufficient:
\begin{definition}\label{def:cplx-trans}
A submanifold $N\xhookrightarrow{\iota} M$ is transversal to a complex lagrangian subbundle $L\subset \TCM$ if
\begin{equation*}
T_\C N+\pr_{T_\C M}L_{|N}=T_\C M_{|N}. 
\end{equation*}
\end{definition}
\begin{lemma}\label{lem:backward-iota}
If $L$ is a complex Dirac structure over $M$ and $N\xhookrightarrow{\iota} M$ is a transversal to $L$, then $\iota^{!}L$ is a complex Dirac structure over $N$.
\end{lemma}

By Lemmas \ref{lem:widehat-backward} and \ref{lem:backward-iota},  if $L$ is a generalized complex structure with associated Poisson bivector $\pi$ and $N\xhookrightarrow{\iota} M$, then $\widehat{\iota^! L}=\iota^! L_{\pi}$. Moreover:

\begin{lemma}\label{sing_cxDirac}
Let $L$ be a complex Dirac structure with constant real index and order, $N\xhookrightarrow{\iota}M$ a submanifold, and $m\in N$. If $T_{m}N\oplus\Delta_{m}=T_{m}M$, then, around $m$, $\iota^{!}L$ is a complex Dirac structure and $\iota^{!}\widehat{L}$ is a Dirac structure given by the graph of a Poisson bivector vanishing at~$m$.
\end{lemma}
\begin{proof}
By the local foliation property \cite[Section 1.5]{dufour2008local}, there exists a neighbourhood $U$ of $m$ where $N$ is transversal to $\widehat{L}$, and hence to $L$ (as $\Delta_\C \subseteq \pr_{\TCM} L_{|N}$). Consequently, $\iota^{!}L_{|U}$ is a complex Dirac structure and $\widehat{\iota^{!}L}_{|U}$ is a Dirac structure. Since $\widehat{\iota^{!}L}_{|U}=\iota
^{!}\widehat{L}_{|U}$ by Lemma \ref{lem:widehat-backward}, the presymplectic leaves of $\widehat{\iota^{!}L}_{|U}$ are the intersection of $N\cap U$ with the presymplectic leaves of $\widehat{L}$. Consequently, the leaf of $\iota
^{!}\widehat{L}_{|U}$ passing through $m$ is a single point (we shrink $U$ if needed) and, by Proposition \ref{sing_Dirac}, the last part follows.
\end{proof}

Finally, we give a definition motivated by Lemma \ref{sing_cxDirac} and Table \ref{table:examples}.
\begin{definition}\label{cr_type_def}
Let $L$ be a complex Dirac structure. We say that a point $m\in M$ is of CR type if, at $m$, the real index and the order are equal and the type is maximal. 
\end{definition}

By \eqref{eq:dim_assoc_dist}, a point $m$ is of CR type if and only if $\Delta_m=\{0\}$. Assume moreover that $\widehat{L}$ is a Dirac structure. By Proposition \ref{sing_Dirac}, $\widehat{L}$ is, around $m$, the graph of a Poisson bivector vanishing at $m$. Consequently, $\Delta_0=\widehat{L}\cap TM=\{0\}$ and so the real index and the order are equal around $m$. In terms of Table \ref{table:examples}, the complex Dirac structure lies in the top right corner at $m$, whereas around $m$ it lies in the top row (of a possibly different table with higher real index).

\begin{remark}
On generalized complex structures, the points of CR type are precisely the points of complex type.
\end{remark}

\section{More examples}

\subsection{Order and type change with constant real index}\label{sec:ex-order-change}

Consider $M=\R^{3}$ with coordinates $(x,y,z)$ and 
the lagrangian subbundle  $$L=L(E,i\iota^{*}_{E}\omega),$$
with $\omega=dx\wedge dy\in\Omega^{2}(\R^{3})$ and $E$ the involutive regular distribution $$E=\la\partial_{x}, (e^{y}\partial_{y}+if(y)\partial_{z})\ra_{\C},$$
where $f\in C^{\infty}(M;\R)$ only depends on the variable $y$ and has non-empty zero set $Z=\{(x,y,z)\:|\:f(y)=0\}$.

Since $E$ is regular, $L$ is smooth.  Moreover, $d_E \iota^*_{E}\omega=\iota^*_{E} d\omega=0$ and by Proposition \ref{prop:E-regular}, $L$ is a complex Dirac structure.

We have $\Delta_{|Z}=\la \partial_{x},\partial_{y}\ra_\R$ and $\Delta_{|M\setminus Z}=\la \partial_{x}\ra_\R$, whereas for $\omega_{\Delta}=\iota_{\Delta}^{*}\omega$,
$$
\omega_{\Delta}|_{p}=\begin{cases}
dx\wedge dy,& \text{for}\: p\in Z,\\
0,&\text{for}\: p\not\in Z.\\
\end{cases}
$$
So $\omega_{\Delta}$ is nondegenerate on $Z$ and has one-dimensional kernel on $M\setminus Z$.

We look at the associated distributions and invariants of $L$. 
\medskip
\setlength{\tabcolsep}{20pt}
\renewcommand{\arraystretch}{1.5}
$$\begin{array}{|c|c|c|c|c|c|} \hline
            &  \Delta &   D    &  \order L  & \type L & \ri L\\ \hline
p\in Z & \la \partial_{x},\partial_{y}\ra_\R & \Delta_{p} & 1 & 0 & 1  \\ \hline
p\not\in Z & \la \partial_{x}\ra_\R & T_{p}M & 0 & 1 & 1 \\    \hline
\end{array}$$
\medskip

The real index follows from $\dim \ker \omega_\Delta$, the order and Proposition \ref{prop:linear-normal-form}.
Thus, $L$ has constant real index one, and type and order changing along $Z$. Note that the distribution $\Delta$ is not smooth and, hence, neither is the associated lagrangian distribution $\widehat{L}$.

 \begin{remark}
 A straightforward computation shows that
  $$K=\R\cdot (f(y)\partial_{x}+e^{y}dz),$$
  which also explains the variation of the invariants along $M$.
 \end{remark}

 \subsection{A foliation with generalized complex leaves}\label{gen_fol}

 Let $D\subseteq TM$ be a regular involutive distribution. Consider a complex Dirac structure $L\subset \TCM$ such that $K=\Ann D$. We then have $K^\perp=D\oplus T^*M$, so $D$ is precisely the distribution defined in Section \ref{sec:order-type}, and the real index and the order of $L$ are constant and equal. Moreover,  
 $$ \frac{K^\perp}{K} = \frac{D\oplus T^{*}M}{\Ann D} \cong D\oplus D^{*},$$
 so $\cJ$ in Proposition \ref{genmap} defines a generalized almost complex structure on each leaf $S\subseteq M$, as $TS=D_{|S}$. We prove next that it is integrable.
 
The Dorfman bracket descends to $D\oplus D^*$, obtaining the bracket
$$[X+\alpha, Y+\beta]_{D}=[X,Y]+\lie^{D}_{X}\beta-\imath_{Y}d_{D}\alpha,$$
for $X+\alpha, Y+\beta\in \Gamma(D\oplus D^{*})$, where $d_{D}$ is the differential along $D$ and 
$$\lie^{D}_{X}=\imath_{X}d_{D}+d_{D}\imath_{X}.$$
From the axioms of the Courant algebroid $\TM$, it follows that:
\begin{lemma}
The bundle $D\oplus D^{*}$ is a Courant algebroid with the natural pairing, bracket and anchor.
\end{lemma}

 The $+i$-eigenbundle of $\cJ$ is 
 $$ L'=\{X+ \xi_{|D_\C} \st X+\xi \in L\}\subset (D\oplus D^*)_\C.$$
 Since $(\lie_{X}\eta-\imath_{Y}d\xi)_{|D_\C}=\lie^{D}_{X}(\eta_{|D_{\C}})-\imath_{Y}d_{D}(\xi_{|D_{\C}})$ and $L$ is involutive, the subbundle $L'$ is involutive. This involutivity descends to any leaf $S\subseteq TM$, as any element of $\Gamma( L'_{|S})$ can be extended to $\Gamma(L')$. We thus have a foliation by generalized complex leaves.

\begin{remark}\label{rmk_gen_CR} 
This example appeared in \cite{li2011av} with the name of a generalized CR structure.
\end{remark}

\section{Splitting theorem}
\label{sec:splitting}

Our last result is an extension of the splitting theorem for generalized complex structures \cite[Thm. 1.4]{abouzaid2006local} to complex Dirac structures. 

\begin{theorem}\label{splitting}
Let $L\subset \TCM$ be a complex Dirac structure with constant real index $r$ and order $s$. Let $m\in M$ be a point of type $k$, denote by $(S, \omega)$ the presymplectic leaf of $\widehat{L}$ passing through $m$, and let $N\subseteq M$ be a submanifold containing $m$ such that $T_{m}N\oplus T_{m}S=T_{m}M$. Then there exist neighbourhoods $V\subseteq N$, $W\subseteq S$ and $U\subseteq M$ of $m$, with $U\cong V\times W$, and a local closed real two-form $B$, such that 
$$L_{|U}\cong e^{B}(\iota^{!}L\times L_{i\omega}),$$
where $\iota:V\to M$ is the inclusion, $\iota^{!}L$ has constant real index $s$ and order $s$ and it is of CR type (Definition \ref{cr_type_def}) at $m$, and $L_{i\omega}$ is the complex Dirac structure associated to $(W,\omega_{|W})$.
\end{theorem}

The pointwise proof of Proposition \ref{prop:linear-normal-form} does not extend locally, since $L/K_{\C}$ does not always define a generalized complex structure. Our proof of Theorem \ref{splitting} is inspired by \cite{bursztyn2019splitting} but requires a local understanding of complex Dirac structures. We start by finding a local isomorphism for $L$ where all its associated distributions are trivial.

Note that $\Delta_0:=\ker \omega_{\Delta}$ is a regular involutive distribution of rank $r-s$ whose associated foliation is the null foliation $\mathcal{F}$ of $\widehat{L}$. Take a foliated chart $U$ where $\mathcal{F}$ is a simple foliation. Since $T_m N \cap \Delta_0=0$, $N$ is transversal to $\mathcal{F}$ (shrinking $U$ if necessary) and we can take a neighbourhood of $m$ in $N$ as the leaf space $P$ with associated submersion $u:U\to P$. Note that $u(U\cap N)=U\cap N$ and from the chart we get a diffeomorphism $\psi_1:U\to P\times \R^{r-s}$.  If $S'$ is a leaf of $\Delta$, then $u(U\cap S')=P\cap S'$ (as $S'$ is foliated by leaves of $\Delta_0$) and so $\psi_1(S')=(P\cap S')\times \R^{r-s}$.  By \cite[Cor. 2.6.3]{courant1990dirac}, the leaf space $P$ inherits a Poisson bivector $\pi$ from the Dirac structure $\widehat{L}$ and $\widehat{L}_{|U}=u^!L_\pi$. 

Let $S$ be the presymplectic leaf of $\widehat{L}$ passing through $m$, note that $u(S)=P\cap S$ is the symplectic leaf of $\pi$ passing through $u(m)$. By the Weinstein Splitting theorem \cite[Thm. 1.4.5]{dufour2006poisson} around $u(m)$ with transversal $U\cap N$, there exists a chart $\psi_2$ on $P$ around $u(m)$ with chart
$$\psi_2=(x_1,\ldots, x_{2k+s}, q_1,\ldots, q_{n-k},p_1,\ldots, p_{n-k}),$$
with coordinates $(x_1,\ldots, x_{2k+s})$ for $U\cap N$ and $(q_1,\ldots, q_{n-k},p_1,\ldots, p_{n-k})$
 for $\R^{2(n-k)}$ (which is equivalent to the symplectic leaf), 
 and in those coordinates $\pi=\pi_{0}+\pi_{N}$ with $\pi_0=\sum_{i=1}^{n-k} \partial_{p_{i}}\wedge\partial_{q_{i}}$ and $\pi_N$ a Poisson bivector in a neighbourhood $V_0$ of $u(m)$ in $U\cap N$ vanishing at $u(m)$. Finally, we obtain a foliated chart for $\Delta_0$, $$\psi=\psi_1^{-1}\circ (\psi^{-1}_2,Id) : \mathcal{U}\to U,$$
 $$\psi^{-1}=(x_1,\ldots, x_{2k+s}, q_1,\ldots, q_{n-k},p_1,\ldots, p_{n-k}, y_1,\dots y_{r-s}),$$
 where $V$ is the image by  $(x_1,\ldots, x_{2k+s})$ of $V_0$ and $\mathcal{U}=V\times \R^{2(n-k)}\times \R^{r-s}$,
 satisfying:
 \begin{enumerate}[a), leftmargin=1.4em]
    \item $\psi$ sends $V\times \{0\}$ onto $V_0$, and $\{m\}\times \R^{2(n-k)}\times \R^{r-s}$ onto $U\cap S$.
     \item $\psi_2 \circ u\circ \psi^{-1}:V\times \R^{2(n-k)}\times \R^{r-s}\to V\times \R^{2(n-k)}$ is the projection.
     \item $\psi_{*}(\{0\}\times \{0\}\times T\R^{r-s})=\Delta_0$.
     \item $\psi^! \widehat{L}$ has a simple null foliation with leafwise Poisson bivector $\pi_0+\pi_N$.
 \end{enumerate}
 
 Recall that, given a vector bundle $E$ over $M$, the \emph{Euler vector field} is the vector field associated to the one-parameter subgroup $s\mapsto \kappa_{e^{-s}}$ in $ \Aut(E)\subseteq \Diff(E)$, where $\kappa_{t}:E\to E$, $\kappa_{t}e=te$ for every $t\in \R\setminus\{0\}$ and $e\in E$. In a fibred local coordinate system for $E$, $(x_{j},y_{j})$, with $x_{j}$ as the fibre directions and $y_{j}$ as the base directions, it is expressed~as
$$\sum_{j}x_{j}\frac{\partial}{\partial x_{j}}.$$

\begin{lemma}\label{splitting_step2}
With the notation above, there exists a section $\mathcal{E}=X+\beta+i\alpha\in \Gamma(\psi^{!}L_{|U})$ such that $\mathcal{E}_{|V\times \{0\}}=0$, $X$ is the Euler vector field of the trivial bundle $\mathcal{U}=V\times \R^{2(n-k)+r-s}$ over $V$,  $\beta\in \Omega^{1}(\mathcal{U})$, and
\begin{equation}\label{eq:alpha}
\alpha=\sum^{n-k}_{j=1}(q_{j}dp_{j}-p_{j}dq_{j})\in \Omega^1 (\mathcal{U}).    
\end{equation}
\end{lemma}

\begin{proof}
We keep the notation $K$, $\Delta_0$, $\widehat{L}$ and $u$ for their respective transformations by $\psi$. The expression \eqref{eq:alpha} defines also $\alpha_{0} \in \Omega^1(V\times \R^{2(n-k)})$, so that $\alpha=u^*\alpha_0$.
Consider the vector field 
$$Y =\sum^{n-k}_{j=1}(p_{j}\partial_{p_{j}}+q_{j}\partial_{q_{j}}) \in \mathfrak{X}(\mathcal{U}).$$
Note that $\pi(\alpha_{0})=\pi_{0}(\alpha_{0})=u_* Y$, so $u_* Y +\alpha_{0}\in L_{\pi}$. Hence, $Y+\alpha\in \widehat{L}$. By \eqref{eq:widehatL}, there exists $\beta_1\in\Omega^1(\mathcal{U})$ such that 
 \begin{equation}\label{eq:section-1}
     Y + \beta_1 + i\alpha \in \Gamma(\psi^! L_{|U}).
 \end{equation}
 
  On the other hand, since $\pr_{TM}K=\Delta_0$, we obtain a local frame \begin{equation*}\label{frame_K}
 \{\partial_{y_{1}}+\zeta_{1},\ldots ,\partial_{y_{r-s}}+\zeta_{r-s},  \zeta_{r-s+1},\ldots ,\zeta_{r}\}
 \end{equation*}
of $K$, where $\zeta_{j}\in \Omega^1(\mathcal{U})$, $\{\zeta_{j}\}_{j=r-s+1}^r$ is a frame of $K\cap T^*\mathcal{U}$ (shrink $U$ if necessary). 
By evaluating \eqref{eq:section-1} on $V\times \{0\}$, we obtain $\beta_{1}{}_{|V\times \{0\}}$, which is a section of $(K\cap T^{*}\mathcal{U})_{|V\times \{0\}}$. Thus, there exist functions $c_{j}\in C^{\infty}(V)$, for $r-s+1\leq j\leq  r$, such that $-\beta_{1}{}_{|V\times \{0\}}=\sum_{j=r-s+1}^{r} c_{j}\zeta_{j}{}_{|V\times \{0\}}$. By extending the functions $c_{j}$ to $\mathcal{U}$, we obtain a section
\begin{equation}\label{eq:section-2}
\beta_{2}=\sum_{j=r-s+1}^{r} c_{j}\zeta_{j}\in \Gamma(K\cap T^{*}\mathcal{U})\subseteq \Gamma( \psi^! L_{|U})
\end{equation}
such that $\beta_{2}{}_{|V\times \{0\}}=-\beta_{1}{}_{|V\times \{0\}}$. Finally, we consider 
\begin{equation}\label{eq:section-3}
    \sum^{r-s}_{j=1} y_j\partial_{y_j}  + \sum^{r-s}_{j=1} y_j \zeta_j \in \Gamma(\psi^! L_{|U}).
\end{equation}
By taking the sum of \eqref{eq:section-1}, \eqref{eq:section-2} and \eqref{eq:section-3}, we obtain  $\mathcal{E}$. 
\end{proof}

 We are now ready to prove Theorem \ref{splitting}.
\begin{proof}[Proof of Theorem \ref{splitting}]

Assume that Theorem \ref{splitting} holds for $\psi^! L$ around $(m,0,0)$ with transversal $V$. Then, 
$$\psi^! L_{|U}\cong e^{B'}(\iota_{V\times \{0\}}^! \psi ^! L\times L_{(\R^{\rk \Delta_m}, \omega)}).$$
By applying $\mathbb{T}\psi$ on both sides we recover the local splitting of $L$. Consequently, for the sake of simplicity, we can use in this proof the notation $M$ for $U$, $N$ for $V$ and $L$ for $\psi^! L_{|U}$, and also the identification $M=N\times \R^{2(n-k)+r-s}$.

Let $\mathcal{E}=X+\beta+i\alpha \in \Gamma(L)$ be as in Lemma \ref{splitting_step2}. Since $X$ is an Euler field, the associated one-parameter subgroup to $\mathcal{E}$ in $\Aut(\TM)$  is
$$\mathbb{T} \kappa_{e^{-s}}\circ e^{\sigma_{s}},$$
where the two-form, by the change of variable $u=\log \tau$, is given by
$$\sigma_{s}=-\int^{s}_{0}\kappa_{e^{u}}^{*}(d\beta+id\alpha)du =- \int_{1}^{e^{s}}\frac{1}{\tau}\kappa_{\tau}^{*}(d\beta+id\alpha)d\tau.$$

Since $\mathcal{E}\in \Gamma(L)$, its flow preserves $L$ (see, e.g., \cite[Lem. 1.1.5]{vanderleer}). By performing the change of variable $s=-\log t$ and acting on $L$,
\begin{equation}\label{eq:L-preserved}
L = e^{\sigma_{\log t}} \mathbb{T} \kappa_{-t} L,   
\end{equation}
for all $t>0$. We define real forms $B_t$, $\omega_t \in \Omega^2(M)$, for $t>0$  such that 
$$B_{t}+i\omega_{t}=\sigma_{\log t}=\int^{1}_{t}\frac{1}{\tau}\kappa_{\tau}^{*}(d\beta+id\alpha)d\tau.$$
Note that this integral is well defined for $t=0$, so we can also define $B_0$ and $\omega_0$. Since $\alpha$ only depends on coordinates $p_j$ and $q_j$ we have 
\begin{equation*}
\omega_{t}=\int^{1}_{t}\frac{1}{\tau}\kappa^{*}_{\tau}(d\alpha)d\tau=(1-t^{2})d\alpha, 
\end{equation*} so $\omega_0=\sum_{j=1}^{n-k} dq_{j}\wedge dp_{j}$. Consider the limit when $t\to 0$ in \eqref{eq:L-preserved}, so
\begin{equation}\label{eq:L-preserved-2}
L= e^{B_0+i\omega_0} (\kappa_0^!L).    
\end{equation}
With $p:M=N\times \R^{2(n-k)+r-s}\to N$ the projection map and $\iota:N\to M$  the inclusion map, we have $\kappa_{0}=\iota\circ p$. Thus, \eqref{eq:L-preserved-2} becomes
$$ L = e^{B_0+i\omega_0}(p^{!}\iota^{!}L)=e^{B_0} e^{i\omega_0}(\iota^{!}L\times T\R^{(2n-k)+r-s})=e^{B_0}(\iota^{!}L\times L_{i\omega_0}).$$
By Lemma \ref{ri_order_adtv}, $\iota^{!}L$ has constant real index $s$ and order $s$. By Lemma \ref{sing_cxDirac}, $\iota^{!}L$ is a complex Dirac structure with associated Poisson structure vanishing at $m$ and so $m$ is of CR type. The result follows with $B=B_0$, which is closed.
\end{proof}

\begin{remark}
The section $\mathcal{E}=X+\beta+i\alpha\in \Gamma(\psi^! L_{|U})$ defines a section $\widehat{\mathcal{E}}=X+\alpha\in \Gamma(\psi^! \widehat{L}_{|U})$. By \cite[Thm. 5.1]{bursztyn2019splitting}, we obtain a splitting for $\widehat{L}$ in terms of the real Dirac structures associated to the factors of the  splitting of $L$, which coincides with that of \cite{blohmann}.
\end{remark}

A regular point for a complex Dirac structure is a point admitting a neighbourhood where the type is constant. For these points we have:

\begin{corollary}\label{regular_splitting}
Let $L\subset \TCM$ be a complex Dirac structure with constant real index $r$ and order $s$; let $m$ be a regular point of type $k$ and $N\subseteq M$ be a $(2k+s)$-dimensional submanifold containing $m$ transversal to $\Delta_m$. Then there exist a neighbourhood $U$ of $m$ such that
$$L|_{U}\cong e^{B}( L_{(D,J)}\times L_{i\omega}),$$
where $L_{(D,J)}$ is associated to a CR structure of codimension $s$ on a neighbourhood of $m$ in $N$, $L_{i\omega}$ is associated to the presymplectic foliation, and $B$ is a real two-form on $M$ that is closed on the presymplectic directions.
\end{corollary}
\begin{proof}
Since $m$ is a regular point, there exists a neighbourhood of $m$ where the type is constant. Let $N$ be any transversal to $\Delta$ at $m$ inside that neighbourhood. Since $T_{m}N\oplus\Delta_m=T_{m}M$ at $m$, by shrinking $N$ if necessary, we have that $TN\oplus\Delta_{| N}=TM_{|N}$.

By Theorem \ref{splitting}, there exist a neighbourhood $U$ of $m$, and a neighbourhood of $m$ in $N$, which we denote again by $N$, such that $U\cong N\times S'\cong N\times \R^{2(n-k)+r-s}$ and $L_{|U}\cong e^{B'}(\iota^{!}L\times L_{i\omega})$ for $B'\in\Omega^2_{cl}(U)$.
 Since $TN\oplus\Delta_{| N}=TM_{|N}$, $\iota^{!}L$ is a complex Dirac structure of constant real index $s$, order $s$ and maximum type. By Proposition \ref{ext_type_cr}, there exists  $B_{1}\in \Omega^{2}(U)$ such that $e^{B_{1}}(\iota^{!}L)$ is a CR structure $(D,J)$ on~$N$. 

Finally, we have that $L_{|U}\cong e^{B'-pr^{*}_{N}B_{1}}(L_{(D,J)}\times L_{i\omega})$, where $\pr_{N}$ is the projection onto $N$. Note that $B=B'-pr^{*}_{N}B_{1}$ is only closed on the directions of $\R^{2(n-k)+r-s}$, the presymplectic foliation. \end{proof}

\appendix
\section*{Appendix: Visual representation of invariants}

The type in generalized complex structures takes integer values in the interval between $0$ (symplectic up to $B$-tranformation) and $\dim M/2$ (complex up to action by a real two-form).

\vspace{.2cm}
\begin{figure}[H]
\resizebox{.9\linewidth}{!}{
 \centering
 \begin{tikzpicture}[
  back line/.style={dashed, line width=0.3mm},
  cross line/.style={densely dotted, , line width=0.3mm},
  front line/.style={line width=0.4mm}]
  \pgfmathsetmacro{\x}{2.78};
  \pgfmathsetmacro{\y}{1.2};
  \pgfmathsetmacro{\factor}{0.3};

  \node (S) {symp.};
  \node [right of=S, node distance=\x cm] (X1) {};
  \node [right of=S, node distance=2*\x cm, label=below:gen.cplx.] (X2) {};
  \node [right of=S, node distance=3*\x cm] (X3) {};
  \node [right of=S, node distance=4*\x cm] (CO) {cplx.};
  \draw[front line, line width=.5mm] (S) -- (CO);
\end{tikzpicture} }
\vspace{-.2cm}
    \caption{Type in generalized complex geometry}
    \label{fig:line}
\end{figure}
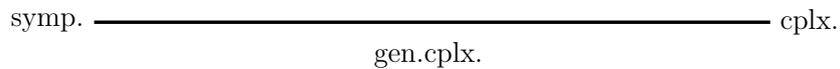
For complex Dirac structures, the real index $r$ ranges between $0$ and $\dim M$, the type between $0$ and $(\dim M-r)/2$ and, by \eqref{eq:order-ri}, the order ranges between $0$ and $r$. From the two latter inequalities, any admissible combination of the invariants can be represented as an integer point of a right tetrahedron (Figure \ref{fig:tetrahedron}, the two right angles being at the vertex `presymp.$_\C$'). We label the vertices with the corresponding structures for constant invariants, up to a transformation by a real two-form, which is closed on the left edge and, in general, closed in the presymplectic directions. The front edge of this tetrahedron corresponds to Figure \ref{fig:line} (where $L\cap \overline{L}=\{0\}$), whereas any slice for fixed real index is Table \ref{table:examples}. The furthest vertical edge corresponds to the complexification of Dirac structures in Example \ref{ex:complex-Dirac-1}, that is, $L=\overline{L}$.

\begin{figure}[htpb]
    \centering
    \resizebox{.75\linewidth}{!}{
 \begin{tikzpicture}[
  back line/.style={dashed, line width=0.3mm},
  cross line/.style={densely dotted, , line width=0.3mm},
  front line/.style={line width=0.4mm}]
  \pgfmathsetmacro{\x}{2.78};
  \pgfmathsetmacro{\y}{1.9};
  \pgfmathsetmacro{\factor}{0.3};

  \node (S) {symp.};
  \node [right of=S, node distance=\x cm] (X1) {};
  \node [right of=S, node distance=2*\x cm, label=below:gen.cplx.] (X2) {};
  \node [right of=S, node distance=3*\x cm] (X3) {};
  \node [right of=S, node distance=4*\x cm] (CO) {cplx.};

  \node [above of=X2, node distance=2*\y cm] (TM) {presymp.$_\C$};
  \node [above of=X2, node distance=4*\y cm] (DC) {Dirac$\phantom{}_\C$};
  \node [above of=X2, node distance=6*\y cm] (T*M) {$T^*_\C M $};

  \node [above of=X1, node distance=\y cm] (PS) {presymp.};
  \node [above of=PS, node distance=\y cm] (PSF) {Dirac}; 
  \node [above of=PSF, node distance=\y cm] (SF) {Poisson};
  
  \node [above of=X3, node distance=\y cm] (TH) {trans.hol.};
  \node [above of=TH, node distance=\y cm] (TCR) {trans.CR};
  \node [above of=TCR, node distance=\y cm] (CR) {CR};
  
  \draw[back line] (S) -- (PS) -- (TM) -- (DC) -- (T*M);
  \draw[back line] (TM) -- (TH) -- (CO);
  \draw[cross line] (PS) -- (TH) -- (TCR) -- (CR) -- (SF)  -- (PSF) -- (PS);
  \draw[front line, line width=.5mm] (S) -- (CO) -- (CR) -- (T*M) -- (SF) -- (S);
  
  \node [right of=S, node distance=.3*\x cm] (X12) {};
  \coordinate [above of=X12, node distance=4.5*\y cm] (AO) {};
  \coordinate [above of=AO, node distance=\factor*\x cm] (AY) {};
  \coordinate [right of=AO, node distance=\factor*\x cm] (AX) {};
  \coordinate [above of=AX, node distance=\factor*\y cm] (AZ) {}; 
  \draw[->] (AO) -- (AX) node[below] {$\type$};
  \draw[->] (AO) -- (AY) node[below left] {$\order$};
  \draw[->] (AO) -- (AZ) node[above right] {$\ri$};
\end{tikzpicture} }
    \caption{Right tetrahedron representing geometric structures encompassed by complex Dirac structures}
    \label{fig:tetrahedron}
\end{figure}
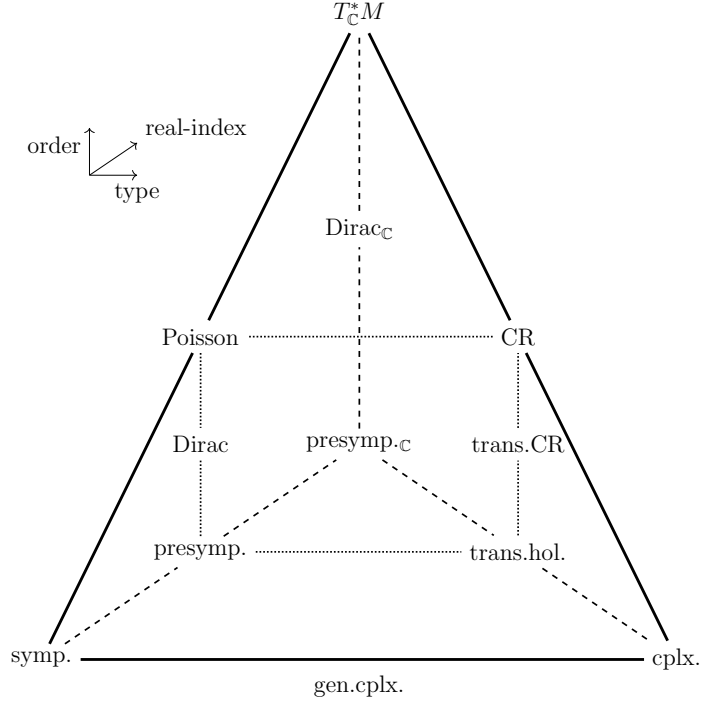

We can interpret the associated Dirac structure $\widehat{L}$ in terms of this tetrahedron and draw the surface of all $L$ mapping to the same type of $\widehat{L}$, which contains $(\widehat{L})_\C$. A point of coordinates $(r,s,k)$ corresponds to complex Dirac structures whose distribution $\Delta$ has rank $$2(n-k)+r-s=(2n+r)-s-2k=\dim M-s-2k,$$ so  the order of $(\widehat{L})_\C$ is $s+2k$. We thus obtain the planes $s+2k=C$ for $0\leq C\leq \dim M$. The intersection with the plane $s+2k=0$ corresponds to the edge (symp.--presymp.$_\C$) whereas the intersection with $s+2k=\dim M$ corresponds to (cplx.--$T_{\C}^*M$).

If we look at the change of invariants, we can regard a complex Dirac structure as moving along the tetrahedron. For instance, the property of a point $m$ of CR type (described at the end of Section \ref{sec:assoc-Dirac}) can be understood as the complex Dirac structure lying at $m$ on a point CR of the top right edge, and around $m$, on the triangle (CR--$T^*_\C M$--Poisson). In fact, the upper semicontinuity of the real index and the order, the parity of the real index and the type, and the upper semicontinuity of the type for constant order, give constraints on how a complex Dirac structure can move within this tetrahedron.

\bibliographystyle{alpha}
\bibliography{references}

\end{document}